\documentclass[11pt,a4paper]{article}

\usepackage{amsmath}
\usepackage{amsthm}
\usepackage{amsfonts}
\usepackage{paralist}
\usepackage{graphicx}
\usepackage{color}
\usepackage{cite}
\usepackage{hyperref}

\hypersetup{%
plainpages=false,
colorlinks,urlcolor=blue,linkcolor=blue,citecolor=blue,
pdftitle={Impact of storage on energy markets},
pdfauthor={James Cruise, Lisa Flatley and Stan Zachary},
pdfsubject={},
pdfpagemode={UseOutlines},
pdfpagelayout={OneColumn},
pdfstartview={FitBH}
}

\parskip\smallskipamount
\newtheorem{theorem}{Theorem}
\newtheorem{proposition}{Proposition}
\theoremstyle{remark}
\newtheorem{remark}{Remark}
\newtheorem{example}{Example}

\newcommand{\R}{\mathbb{R}}

\newcommand{\X}{X}

\parskip \smallskipamount

\title{Impact of storage competition on energy markets}
\author{James Cruise\footnote{Heriot-Watt University.  Research
    supported by EPSRC grant EP/I017054/1},
   Lisa Flatley\footnote{University of Warwick}
  \ and Stan Zachary\footnotemark[1]}
\date{\today}

\begin{document}

\maketitle

\begin{abstract}
  We study how storage, operating as a price maker within a market
  environment, may be optimally operated over an extended period of
  time.  The optimality criterion may be the maximisation of the
  profit of the storage itself, where this profit results from the
  exploitation of the differences in market clearing prices at
  different times.  Alternatively it may be the minimisation of the
  cost of generation, or the maximisation of consumer surplus or
  social welfare.  In all cases there is calculated for each
  successive time-step the cost function measuring the total impact of
  whatever action is taken by the storage.  The succession of such
  cost functions provides the information for the storage to determine
  how to behave over time, forming the basis of the appropriate
  optimisation problem.  Further, optimal decision making, even over a
  very long or indefinite time period, usually depends on a knowledge
  of costs over a relatively short running time horizon---for storage
  of electrical energy typically of the order of a day or so.

  We study particularly competition between multiple stores, where the
  objective of each store is to maximise its own income given the
  activities of the remainder.  We show that, at the Cournot Nash
  equilibrium, multiple large stores collectively erode their own
  abilities to make profits: essentially each store attempts to
  increase its own profit over time by overcompeting at the expense of
  the remainder.  We quantify this for linear price functions

  We give examples throughout based on Great Britain spot-price market
  data.
\end{abstract}

\section{Introduction}
\label{sec:introduction}
There has been much discussion in recent years on the role of storage
in future energy networks.  It can be used to buffer the highly
variable output of renewable generation such as wind and solar power,
and it further has the potential to smooth fluctuations in demand,
thereby reducing the need for expensive and carbon-emitting peaking
plants.  For a discussion of the use of storage in providing multiple
buffering and smoothing capabilities, including the ability to
integrate renewable generation into energy networks see, for example,
the fairly recent review by Denholm et al (2010)~\cite{DEKM}, and the
many references therein.  Within an economic framework much of the
value of energy storage may be realised by allowing it to operate in a
market environment, provided that the latter is structured in such a
way as to allow this to happen.  Thus the smoothing of variations in
demand between, for example, nighttime when demand is low and daytime
when demand in high may be achieved by allowing a store to buy energy
at night when the low demand typically means that it is relatively
cheap, and to sell it again in the day when it is expensive.
Similarly, the use of storage for buffering against shortfalls in
renewable generation may---at least in part---be effected by allowing
storage to operate in a responsive spot-price market when prices will
rise at the times of such shortfall.  We remark though that if it is
intended that the use of storage should facilitate, for example, a
reduction in carbon emissions, then there is of course no guarantee
that a market environment will in itself permit this to happen; it may
be necessary that the market itself, and the rules under which it
operates, are correctly structured so as to penalise or prohibit
environmentally damaging generation or to reward clean energy
production---for some recent insights into the possible unexpected
side effects of storage operating in a market, see Virasjoki et al~\cite{VRSS}.

A small store may be expected to function as a price-taker, buying and
selling so as, for example, to maximise its own profit over time.
However, a larger store will act as a price-maker, perhaps
significantly affecting the market in which it operates, and thus also
affecting quantities such as generator costs, consumer surplus and
social welfare.  Further a number of larger stores, by competing with
each other, may smooth prices to the point where they are unable to
make sufficient profits as to be economically viable.

Aspects of many of these issues have been explored in the literature.
Recent work on the use of storage in a specifically market environment
is given by Gast et al~\cite{GTL,GLTP}, Graves et al~\cite{GJM}, Hu et
al~\cite{HCBJ} and Secomandi~\cite{Sec2010}.  Sioshansi et at~\cite{SDJW}
study the effects of storage on producer and consumer surplus and on
social welfare.  Sioshansi~\cite{Sio2014} gives an example where storage may
may reduce social welfare.  Gast et el~\cite{GLTP} show how in appropriate
circumstances storage may be used to minimise generation costs and
thus maximise consumer welfare.


In the present paper we aim to develop a more comprehensive
mathematical theory of the way in which storage interacts with the
market in which it operates.  Our fundamental assumption is that each
individual store operates over an extended period of time in such a
way as to optimise its ``profit''---or equivalently minimise its
costs---with respect to time-varying cost functions presented to it.
These may represent either the prevailing costs within a free market,
as may be natural when the store is independently owned, or adjusted
costs which take into account the wider impact of the stores
activities, as would be appropriate when the store was owned, for
example, by the generators or by society---see
Section~\ref{sec:variant-problems}.  Thus if it is desirable that a
store should function in a particular way---for example, to minimise
generation costs---it may be fed the appropriate cost signals and,
given those signals, left to perform as an autonomous agent.  Such an
approach is notably desirable in facilitating distributed control and
optimisation within a possibly complex environment.  In this paper are
particularly interested in studying the economic effects of
competition between multiple stores, not least on the viability of the
stores themselves.  The typically high capital costs of storage, in
relation to operating costs, mean that competition between stores may
reduce price differentials across time to the extent that stores are
unable to make sufficient operating profit as to permit the recovery
of their capital costs.

Within our analysis we therefore treat storage as generating its
revenue by arbitrage within a market in which prices are low at times
of energy surplus and high at times of scarcity.  While, within an
appropriately structured and responsive market, this may allow storage
to operate so as to realise many of its economic benefits, we
acknowledge that there are many other uses of storage whose benefits
may not be so easily captured.  Notably this is the case where storage
needs to react on a very short time scale, for example to compensate
for sudden shortfalls of generation, or to provide stability within a
system, and where there is insufficient time for the value of such
actions to be captured within a spot market environment.  For some
work on the simultaneous use of storage for both arbitrage and
buffering against the effects of sudden events see Cruise and
Zachary~\cite{CZ}, while for work on a whole systems assessment of the
value of energy storage see Pudjianto \emph{et al}~\cite{PADS}.

We outline in Section~\ref{sec:model} the model for the market in
which storage operates.  In particular this allows for supply and
demand which are sensitive to price, and hence also for an impact on
price of the market activities of the storage itself (so that the
storage may be considered as a price maker).  We assume for the moment
(but see also below) that a single store wishes to optimise its own
profit, or minimise its own costs, by trading in the market, we
formulate the corresponding optimisation problem faced by the store
and we state how it may be solved.  Formally the environment is
deterministic, but we discuss also the extent to which it is possible
to proceed similarly in a stochastic environment.

In Section~\ref{sec:single-store-market} we study the effect of a
single profit-maximising store in a market.  We look at its effect on
both market prices and on consumer surplus and give sensitivity
results for the variation of the size of the store.  We give examples
based on Great Britain market data.

In Section~\ref{sec:n-competing-stores} we study a number of competing
stores operating in a market.  We consider possible models of
competition, whereby the stores make bids and clearing prices in the
market are determined.  We identify Nash equilibria for the model of
competition in which stores bid \emph{quantities}---a generalisation
of Cournot competition---give existence and uniqueness results, and
show how equilibria may be determined.  We further show that, even for
this arguably most favourable model of competition (from the point of
view of the stores) an oversupply of storage capacity leads to a
situation in which, with linear price functions, the \emph{total}
profit made by all the stores is approximately inversely proportional
to their number.  Essentially what happens here is that, relative to a
cooperative solution, each store over-trades in order to acquire a
larger share of total profit, thereby impacting on the market in such
a way as to reduce price differentials over time and thus also the
profits to be made by other stores.  Thus a sufficiently large number
of stores are unable to make profits, and so---presumably---recoup
their capital costs.  In this section we also give examples again
based on GB market data and relating to such competition between
stores.

Finally, in Section~\ref{sec:variant-problems} we consider variant
problems in which storage (instead of consisting of independent
profit-maximising entities) is managed, for example, for the optimal
benefit of consumers, or for the optimal benefit of generators.  We
show that, by suitable redefinition of cost functions, these variant
problems may be reduced mathematically to those already studied.

\section{Model}
\label{sec:model}

We now formulate our model for a set of $n\ge1$ stores operating in an
energy market.  Formally we treat prices and costs as deterministic.
However, in a stochastic environment it may be reasonable, at each
successive point in time, to replace future prices and costs by their
expected values and to then proceed as in the deterministic case.
That this can, in many cases, lead to optimal or near optimal
behaviour for a single store is shown in Cruise \emph{et
  al}~\cite{CFGZ}.  It is further the case that, for many
applications---notably electricity storage---optimal decision making
over long or even indefinite time horizons nevertheless only requires
a real-time knowledge of future costs over a short running time
horizon, something which is again shown formally in~\cite{CFGZ}.  Thus
electricity storage may make its profits by exploiting differences
between daytime and nighttime prices and if these are sufficiently
different that the storage typically fills and empties on a daily---or
almost daily basis---then ongoing optimal management may never require
a knowledge of future prices for more than a few days ahead.

We assume that each store $j$ has an energy capacity $E_j$ and input
and output rate constraints $P_{Ij}$ and $P_{Oj}$ respectively (the
maximum amount of energy which can enter or leave the store per unit
time).  Each such store $j$ also has an \emph{efficiency}
$\epsilon_j\in(0,1],$ where $\epsilon_j$ is the number of units of
energy output which the store can achieve for each unit of energy
input.  We assume without loss of generality that any loss of energy
due to inefficiency occurs immediately after leaving the store (so
that the above capacity and rate constraints---both input and
output---apply to volume of energy input).  For simplicity we also
assume that there is no time-dependent leakage of energy from the
stores; the simple adjustments required to deal with any such leakage
are analogous to those described in~\cite{CFGZ}.

We work in discrete time $t=1,\dots,T$ for some finite time horizon
$T$.  Associated with each such time $t$ is a price function $p_t$
such that $p_t(x)$ is the market price per unit of energy when $x$ is
the total amount (positive or negative) of energy bought from the
market by all the stores, i.e.\ $xp_t(x)$ is the total cost to the
stores of buying this energy.  (Each of the functions~$p_t$ is of
course influenced by everything else that is happening in the market
at time~$t$; it explicitly measures only the further effect on price
of the activity of the stores.)  We assume throughout that, over the
range of possible values of its argument (i.e.\ the interval
$[-\sum_{j=1}^n\epsilon_jP_{Oj},\sum_{j=1}^nP_{Ij}]$), each of the
functions $p_t$ is positive and increasing and is such that, for any
constant $k$, the function of $x$ given by $xp_t(x+k)$ is convex and
increasing.  (The quantity $xp_t(x+k)$ is the total cost to a store of
buying $x$ units of energy---again positive or negative---at time~$t$
when the total amount bought by the remaining stores at that time is
$k$.)  An important case in which these conditions are satisfied, and
which we consider in detail later, is that where the prices are
linearised so that
\begin{equation}
  \label{eq:1}
  p_t(x) = \bar p_t + p'_t x
\end{equation}
where $\bar p_t>0$ and where $p'_t\ge0$ is such that the
function~$p_t$ remains positive for all possible values of its
argument as above.  This should, for example, be a good approximation
whenever the total storage capacity is not too large in relation to
the total size of the market in which the stores operate.  In such a
case, we may take $\bar p_t=p_t(0)$ (i.e.\ the price at time $t$
without storage on the system) and $p_t'=p_t'(0)$.  More generally,
the above conditions on the functions $p_t$ seem likely to be
satisfied in many cases, for example when they do not differ too much
from the above linear case, and are in all cases readily checkable.

In particular if $s_t(p)$ is the amount externally supplied to the
market at time $t$ and price $p$ and $d_t(p)$ is the corresponding
total demand at that time and price---and if the functions~$s_t$ and
$d_t$ are given independently of the activities of any stores---then
we may define the residual supply function $R_t$ at that time by
$R_t(p)=s_t(p)-d_t(p)$; if $R_t$ is continuous and strictly increasing
then we have that $p_t$ is the inverse of the function $R_t$ and is
similarly continuous and strictly increasing. If, furthermore, each of
the functions~$R_t$ is differentiable and prices take the
form~(\ref{eq:1}), with $\bar p_t=p_t(0)$ and $p_t'=p_t'(0),$ then we
may relate $p_t'$ to the point elasticities of supply and demand at
price $\bar p_t$, denoted $e_s$ and $e_d$ respectively, in the
following way:
\begin{align}
  \label{eq:2}
  p_t'=\frac{\bar p_t}{e_s s_t(\bar p_t)-e_d d_t(\bar p_t)}.
\end{align}
This method of determining the price functions $p_t$ is especially
relevant when the other players in the market make their decisions
without taking the stores' actions into account, perhaps due to the
relatively small level of storage capacity in relation to the rest of
the market.  With sufficient information, more complex price functions
$p_t$ could be derived, for example by considering games between the
stores and the rest of the energy system.

We denote the successive levels of each store~$j$ by a vector
$S_j=(S_{j0},\dots,S_{jT})$ where each $S_{jt}$ is the energy level of the
store at time~$t$.  It is convenient to assume that the initial and
final levels of the store are constrained to fixed values $S^*_{j0}$
and $S^*_{jT}$ respectively.  For each such vector~$S_j$ and for each
$t=1,\dots,T$, define also $x_t(S_j)=S_{jt}-S_{j,t-1}$ to be the
amount (positive or negative) by which the level of the store is
increased at time~$t$.

In order to incorporate efficiency, it is helpful to define, for each
store $j$, the function $h_j$ on $\R$ by $h_j(x)=x$ for $x\ge0$ and
$h_j(x)=\epsilon_j x$ for $x<0$.  For each time $t$ such that
$x_t(S_j)\ge0$, store~$j$ buys $x_t(S_j)$ units of energy from the
market, while for $t$ such that $x_t(S_j)<0$, it sells
$-\epsilon_jx_t(S_j)$ units of energy to the market.  For each
store~$j$ and time $t$, and given the changes $x_{it}$, $j\neq i$,
(positive or negative) in the levels of the remaining stores at that
time, define now the \emph{cost
  function}~$C_{jt}(\,\cdot\,;\,x_{it},j\neq i)$ by
\begin{equation}
  \label{eq:3}
  C_{jt}(x_{jt};\,x_{it},j\neq i)
    = h_j(x_{jt})p_t\bigg(\sum_{i=1}^nh_i(x_{it})\bigg);
\end{equation}
this represents the cost to store~$j$ of increasing its level by
$x_{jt}$ (again positive or negative) at time~$t$, given the
corresponding activities of the remaining stores at that time.  Note
that the conditions on the function~$p_t$ ensure that
$C_{jt}(x_{jt};\,x_{it},j\neq i)$ is an increasing convex function
of its principal argument $x_{jt}$ and takes the value zero when this
argument is zero.



In particular if the objective of store $j$ is to optimise its profit,
given the \emph{policy} over time $S_i=(S_{i0},\dots,S_{iT})$ of every
other store $i\neq j$, then it faces the following optimisation problem:
\begin{compactenum}[]
\item $\mathbf{P}_j$: Choose $S_j=(S_{j0},\dots,S_{jT})$ so as to
  minimise the function of $S_j$ given by
  \begin{equation}
    \label{eq:4}
    \sum_{t=1}^T C_{jt}(x_t(S_j);\,x_t(S_i),j\neq i)
  \end{equation}
  subject to the capacity constraints
  \begin{gather}
    \label{eq:5}
    S_{j0}=S^*_{j0}, \qquad S_{jT}=S^*_{jT}, \qquad 0 \le S_{jt}\le
    E_j, \quad 1 \le t \le T-1.
  \end{gather}
  and the rate constraints
  \begin{equation}
    \label{eq:6}
    x_t(S_j) \in \X_j,
    \qquad 1 \le t \le T,
  \end{equation}
  where $X_j=\{x:-P_{Oj}\leq x\leq P_{Ij}\}$.
\end{compactenum}
Note that the observed convexity of the cost
functions~$C_{jt}(\,\cdot\,;\,x_{it},j\neq i)$ ensures that a solution
to the optimisation problem~$\mathbf{P}_j$ always exists.

At various points we make use of the following result, taken
from~\cite{CFGZ}, and in which each of the vectors $\mu^*_j$ is
essentially a vector of (cumulative) Lagrange multipliers.

\begin{proposition}
  \label{prop:1}
  For any store $j=1,\dots,n$, and for any fixed policies $S_i$ of
  every other store $i\neq j$, suppose that there exists a vector
  $\mu_j^*=(\mu^*_{j1},\dots,\mu^*_{jT})$ and a value
  $S_j^*=(S^*_{j0},\dots,S^*_{jT})$ of $S_j$ such that
  \begin{compactenum}[(i)]
  \item $S_j^*$ is feasible for the stated problem $\mathbf P_j$;
  \item for each $t$ with $1\le t\le T$, $x_t(S_j^*)$ minimises
    \begin{displaymath}
      C_{jt}(x_{jt};\,x_t(S_i),j\neq i)-\mu^*_{jt}x_{jt}
    \end{displaymath}
    in $x_{jt}\in\X_j$; and
  \item the pair $(S_j^*,\mu_j^*)$ satisfies the complementary slackness
    conditions, for $1\le t\le T-1$,
    \begin{equation}
      \label{eq:7}
      \begin{cases}
        \mu^*_{j,t+1} = \mu^*_{jt} & \quad\text{if $0 < S^*_{jt} < E_j$,}\\
        \mu^*_{j,t+1} \le \mu^*_{jt} & \quad\text{if $S^*_{jt} = 0$,}\\
        \mu^*_{j,t+1} \ge \mu^*_{jt} & \quad\text{if $S^*_{jt} = E_j$.}
      \end{cases}
    \end{equation}
  \end{compactenum}
  Then $S_j^*$ solves the above optimisation problem $\mathbf{P}_j$.
  Further, the given convexity of the cost
  functions~$C_{jt}(\,\cdot\,;\,x_t(S_i),j\neq i)$ guarantees the
  existence of such a pair $(S^*_j,\mu^*_j)$.
\end{proposition}

In the case of a single store, \cite{CFGZ} provides an algorithm which
determines a suitable pair $(S_1^*,\mu_1^*)$ satisfying the conditions
(i)--(iii) above.  A key advantage of the algorithm is its
exploitation of the result that the optimal decision of a store at any
each successive time~$t$ typically depends only on the price
information associated with a relatively short interval of time
subsequent to $t$.  The convexity of the cost functions is required
only to guarantee the existence of such a pair, but as long as such a
pair exists, the algorithm could be implemented (with some obvious
adjustments) to determine the optimal policy of the store under more
general cost functions---see Flatley \emph{et al}~\cite{FMW} for a
discussion of this.  In Section~\ref{sec:n-competing-stores} we adapt
the algorithm in~\cite{CFGZ} to the case of $n$ competing stores.

\begin{remark}
  In cases where the stores are not independent profit maximising
  entities but are instead owned by, for example, the generators or by
  society, the above cost functions~$C_{jt}$ may be appropriately
  modified so that the problems~$\mathbf{P}_j$ continue to define
  optimal behaviour for the stores; see
  Section~\ref{sec:variant-problems} for a discussion of how this may
  be done.
\end{remark}

\section{The single store in a market}
\label{sec:single-store-market}

In the case $n=1$ of a single store it is convenient to drop the
subscript~$j$ and to write $S$ for $S_j$, etc.
The single-store optimisation problem is then to choose
$S=(S_0,\dots,S_T)$ so as to minimise
\begin{displaymath}
  \sum_{t=1}^T C_t(x_t(S))
\end{displaymath}
(where the $C_t$ are the cost functions defined by \eqref{eq:3})
subject to the capacity constraints \eqref{eq:5} and rate
constraints~\eqref{eq:6}.


For simplicity we assume the strict convexity of the cost functions
$C_t$---as, for example, will be the case when the linear
approximation~\eqref{eq:1} holds with $p'_t>0$ for each $t$.  This
strict convexity is sufficient to guarantee the uniqueness of the
solution $S^*$ of the optimisation problem~$\mathbf{P}$.


\subsection{Sensitivity of store activity to capacity and rate
  constraints}
\label{sec:sensitivity-results}

Let $(S^*,\mu^*)$ be the pair identified in Proposition~\ref{prop:1},
defining the solution $S^*$ of the above optimisation
problem~$\mathbf{P}$.  Then the market clearing price at each time~$t$
is $p_t(h(x_t(S^*))$.  The successive clearing prices then determine
such quantities as consumer surplus---in the way we describe later.

As a measure of the sensitivity of the market to variation of the size
of the store, we use Proposition~\ref{prop:1} to describe briefly how
variation of either the capacity or the rate constraints of the store
impacts on the solution~$S^*$ of $\mathbf{P}$.
Proposition~\ref{prop:1} continues to hold when we allow either the
capacity or the rate constraints of the store to depend on the
time~$t$.  Therefore it is sufficient to consider the effect of
variation of these constraints at any single time~$t_0$.

Consider first the effect of an arbitrarily small 
increase (positive or negative) $\delta E_{t_0}$ in the capacity of
the store at time~$t_0$; since the initial and final levels~$S_0^*$
and $S^*_T$ are fixed we assume $0<t_0<T$.  It is clear from
Proposition~\ref{prop:1} that this infinitesimal change has no effect
on $S^*$ unless $S^*_{t_0}=E$; further if $\delta E_{t_0}>0$ we also
require the strict inequality $\mu^*_{t_0+1}>\mu^*_{t_0}$.  Under
these conditions there exist times $t_1<t_0<t_2$, such that the effect
of the increment~$\delta{}E_{t_0}$---provided it is indeed
sufficiently small---is to change $\mu^*_t$, and so also $x_t(S^*)$
(via the condition~(ii) of Proposition~\ref{prop:1}), for $t$ such
that $t_1<t\le t_0$, both the original and the new values of $\mu^*_t$
being constant over this interval, and to similarly change $\mu^*_t$
and $x_t(S^*)$ for $t$ such that $t_0<t\le t_2$, again both the
original and the new values of $\mu^*_t$ being constant over this
interval; all changes within the second of the above intervals have
the opposite sign to those within the first; for all remaining values
of $t$, the parameter~$\mu^*_t$ remains unchanged.  The change in
$\mu^*_t$ over each of the above intervals is readily determined by
the requirement that now $S^*_{t_0}=E+\delta E_{t_0}$.  (Thus, for
example, for a perfectly efficient store and twice differentiable cost
functions~$C_t$, the effect of an increment $\delta E_{t_0}>0$---where
$t_0$ is such that $\mu^*_{t_0+1}>\mu^*_{t_0}$---will be to increase
$x_t(S^*)$ in proportion to $1/C_t''(x_t(S^*))$ for times~$t$ such
that $t_1<t\le t_0$ and at which the input rate constraint is
nonbinding, and to similarly decrease $x_t(S^*)$ in proportion to
$1/C_t''(x_t(S^*))$ for times~$t$ such that $t_0<t\le t_2$ and at
which the output rate constraint is nonbinding.)

Similarly an arbitrarily small 
change at time~$t_0$ in either the input or
the output rate constraint has no effect on $(S^*,\mu^*)$ unless
$\mu^*_{t_0}$ and $x_{t_0}(S^*)$ are such that that constraint is
binding in the solution of the minimisation problem of (ii) of
Proposition~\ref{prop:1}.  The effect is then again to change
$\mu^*_t$ and $x_t(S^*)$ for those $t$ in an interval which includes
$t_0$; both this interval and the required changes are again readily
identifiable from that proposition.

\subsection{Impact of a store on prices and consumer surplus}
\label{sec:impact-store-prices}

\paragraph{Impact on prices.}
\label{sec:impact-prices}

In general we may expect the impact of the store on the market to be
that of smoothing prices over time: the store will in general buy at
times when prices are low, thereby competing in the market and
increasing prices at those times, and similarly sell at times when
prices are high, thereby decreasing them at those times.  Relaxing the
power rates or capacity constraints of the store may then be expected
to result in further smoothing of the prices, as the store is able to
buy and sell more at times of low and high prices, thereby augmenting
the above effect.  We might also expect that increasing the efficiency
of the store will further smooth prices, but this is not so clear-cut,
as we illustrate in the following example.

\begin{example}
  Consider price functions of the linear form~(\ref{eq:1}) and a store
  which operates over just two time steps ($T=2$), starting and
  finishing empty but not otherwise subject to capacity or rate
  constraints.  Suppose further that $p_2=ap_1>0$ for some $a>1$.
  Then, for efficiency~$\epsilon$, the store buys $x(\epsilon)$ units
  of energy at time~$1$ and sells $\epsilon x(\epsilon)$ units at
  time~$2$, where
  \begin{equation}
    \label{eq:8}
    x(\epsilon)=
    \begin{cases}
      0 & \text{if} \quad \epsilon a<1\\
      \dfrac{\epsilon p_2-p_1}{2(p_1'+\epsilon^2 p_2')} &
      \text{otherwise}.
    \end{cases}
  \end{equation}
  In the presence of the store the difference between the market
  clearing price at time~$t_2$ and that at time~$t_1$ is given by
  $p_2(\epsilon x(\epsilon))-p_1(x(\epsilon))$, and it is easy to
  check that for suitable values of the parameters $\bar p_t$, $p'_t$,
  $t=1,2$, this expression is an increasing function of $\epsilon$ for
  $\epsilon$ sufficiently close to $1$---contrary to the expectation
  mentioned above.
\end{example}




\paragraph{Impact on consumer surplus.}
\label{sec:impact-cons-surpl}

The \emph{consumer surplus} associated with a demand function~$d$ and
clearing price~$p_0$ is usually defined as
$\int_{p_0}^\infty d(p)\,dp$, and so the consumer surplus of the
store's optimal strategy~$S^*$ is given by
\begin{equation}
  \label{eq:9}
  \sum_{t=1}^T\int_{p_t(h(x_t(S^*)))}^{\infty}d_t(p) \ dp,
\end{equation}
where $d_t(p)$ is the consumer demand associated with price $p$ at
time $t$.  
If the size or activity level of the store is such that the price
changes caused by its introduction are relatively small, and we
additionally make the linear approximation~\eqref{eq:1}, then the
change in consumer surplus due to the introduction of the store is
well approximated by
\begin{equation}
  \label{eq:10}
  -\sum_{t=1}^T h(x_t(S^*))p'_td_t(\bar p_t).
\end{equation}
It might reasonably be expected that, if the store is reasonably
efficient ($\epsilon$ is close to one) and if prices are
well-correlated with demand, then the store will buy ($x_t>0$) at
times of low consumer demand and sell ($x_t<0$) at times of high
consumer demand, and that this will have a beneficial effect on
consumer surplus---as suggested by~\eqref{eq:10} whenever the price
sensitivities~$p'_t$ are sufficiently similar to each other.  However,
these price sensitivities $p_t'$ do need to be taken into account.
Again we give an example.

\begin{example}
  Consider again a store with linear prices of the form~(\ref{eq:1}),
  which starts and finishes empty and which operates over just two
  time steps, i.e.\ $T=2$.  Assume that the power ratings of the store
  exceed its capacity and that demand is completely inelastic, so
  that, for $t=1,2,$ there exists $d_t^*\geq 0$ such that
  $d_t(p)=d_t^*$ for all prices $p$.  Then, from \eqref{eq:10}, as
  long as $p_1<\epsilon p_2$, the change in consumer surplus on
  introducing the store to the electricity network is
  \begin{align*}
    \min
    \left(\frac{\epsilon p_2-p_1}{2(p_1'+\epsilon^2 p_2')},E\right)
    \left(\epsilon p_2'd_2^*-p_1'd_1^*\right),
  \end{align*}
  which is clearly negative whenever $\epsilon p_2'd_2^*<p_1'd_1^*$.
  In the latter case the price sensitivity~$p'_1$ at time~$1$ is
  sufficiently high that the decrease in consumer surplus at this time
  as a result the store buying outweighs the increase in consumer
  surplus at time~$2$ as a result of the store selling.  Sioshanshi
  \cite{Sio2014} gives similar examples of cases where storage reduces
  \textit{social} welfare, defined as a sum of consumer surplus,
  producer surplus and the store's profit.
\end{example}

\begin{remark}
  In the case of linearised prices of the form~(\ref{eq:1})---so that
  the cost functions $C_t$ are quadratic with a discontinuity of slope
  at $0$---we can deduce some further results.  In particular, both
  the market clearing price at each time~$t$, given by
  $p_t+p'_th(x_t(S^*))$, and the consumer surplus, given by the
  approximation~\eqref{eq:10}, are then piecewise linear functions of
  the capacity of the store.  This follows from the observations of
  Section~\ref{sec:sensitivity-results}, in particular from the
  condition~(ii) of Proposition~\ref{prop:1}, which shows that the
  vector of optimised levels~$S^*$ is a piecewise linear function of
  the vector~$\mu^*.$ As the capacity $E$ is varied at a single time
  $t_0,$ the discussion of Section~\ref{sec:sensitivity-results}
  therefore implies that $\mu^*$ must vary piecewise linearly with
  respect to this variation, between the times $t_1$ and $t_2$
  identified above.
\end{remark}



\subsection{Example}
\label{sec:example}

We consider an example based on half-hourly market electricity prices
in Great Britain throughout the year 2014.  These are the so-called
Market Index Prices as supplied by Elexon~\cite{Elexon}, who are
responsible for operating the Balancing and Settlement Code for the
Great Britain wholesale electricity market.  These are considered to
form a good approximation to real-time spot prices.

These prices, given in units of pounds per megawatt-hour, exhibit an
approximately cyclical behaviour, being high by day and low by night
and, apart from this, are reasonably consistent throughout the year
except for some mild seasonal variation, notably that prices are
slightly lower during the summer months.

We take the price functions~$p_t$ to be given by
\begin{equation}
  \label{eq:11}
  p_t(x)=\bar p_t\left(1+\lambda x\right),
\end{equation}
where the $\bar p_t$, $t=1,\dots T$, are proportional to the spot
market prices referred to above.  These price functions are a special
case of the linear functions~\eqref{eq:1}, in which the price
sensitivity $p'_t$ is proportional to $\bar p_t$, an assumption which
is in many circumstances very plausible; the constant of
proportionality $\lambda\ge0$ may then be considered a \emph{market
  impact factor}.  The relation~\eqref{eq:11} also implies that
$\lambda$ should be chosen in proportion to the physical size of the
unit of energy: for any $k>0$, the substitution of $x/k$ for $x$ and
$k\lambda$ for $\lambda$ leaves~\eqref{eq:11} unchanged.  We therefore
find it convenient to consider a store whose nominal dimensions are
generally held constant, and to allow $\lambda$ to vary: the market
impact as $\lambda$ is increased is equivalent to that which occurs
when $\lambda$ is held constant and the dimensions of the store are
allowed to increase instead.  The case $\lambda=0$ corresponds to no
market impact (appropriate to a relatively small store).  Clearly also
there exists $\lambda_{\max}$ such that, for
$\lambda\ge\lambda_{\max}$ both the rate and capacity constraints of
the store cease to be binding, so that for all
$\lambda\ge\lambda_{\max}$ the market impact of the store is the same,
and---again by the above scaling argument---may be regarded as that of
an unconstrained store.

We take a storage facility with common input and output rate
constraints and, without loss of generality, we choose units of energy
such that, on the half-hourly timescale of the spot-price data, this
common rate constraint is equal to~$1$ unit per half-hour.  For the
numerical example, we in general take the capacity of the store to be
given by $E=10$ units; this corresponds to the assumption that the
store empties or fills in a total time of $5$ hours.  This capacity to
rate ratio is fairly typical, being in particular close to that for
the Dinorwig pumped storage facility in Snowdonia~\cite{Din}
(though the charge time and discharge times for Dinorwig are
approximately 7 hours and 5 hours respectively).  We in general take
the round-trip efficiency as $\epsilon=0.75$, which is again
comparable to that of Dinorwig.  Thus the effect on market prices
given by varying $\lambda$, which we discuss below, corresponds to
that considering the effect on the market of rescaled versions of a
facility not too dissimilar from Dinorwig.  We also investigate
briefly the effect of varying the capacity constraint~$E$ relative to
the unit rate constraint, and the effect of varying the round-trip
efficiency~$\epsilon$.

Figure~\ref{fig:lam} shows, for $E=10$ and $\epsilon=0.75$, the effect
of varying the market impact~$\lambda$.  The control of the store is
optimised, as previously discussed, over the entire one-year period
for which price data are available (with the store starting and
finishing empty).  For relatively small values of $\lambda$ the store
fills and empties (or nearly so) on a daily cycle, as it takes
advantage of low nighttime and high daytime prices.  For significantly
larger values of the market impact factor~$\lambda$, the store no
longer fills and empties on a daily basis (as this factor now erodes
the day-night price differential as the volume traded increases);
however, the level of the store may gradually vary on a much longer
time scale as the store remains able to take advantage of even modest
seasonal price variations.  The first six panels of panels of
Figure~\ref{fig:lam} show plots of the time-varying levels of the
store against selected values of~$\lambda$.  For $\lambda=0$,
$\lambda=0.1$ and $\lambda=0.5$ the level of the store is plotted
against time for the first two weeks of the year, while for
$\lambda=1$, $\lambda=5$ and $\lambda=10$ the level of the store is
plotted against time for the entire year.  The final panel of
Figure~\ref{fig:lam} shows a plot against time---for the first two
weeks of the year---of the market clearing price corresponding to
$\lambda=0$, $\lambda=0.5$, and $\lambda=10$.  The erosion of the
day/night price differential as $\lambda$ increases is clearly seen.

\begin{figure}[!ht]
  \centering
  \includegraphics[scale=0.73]{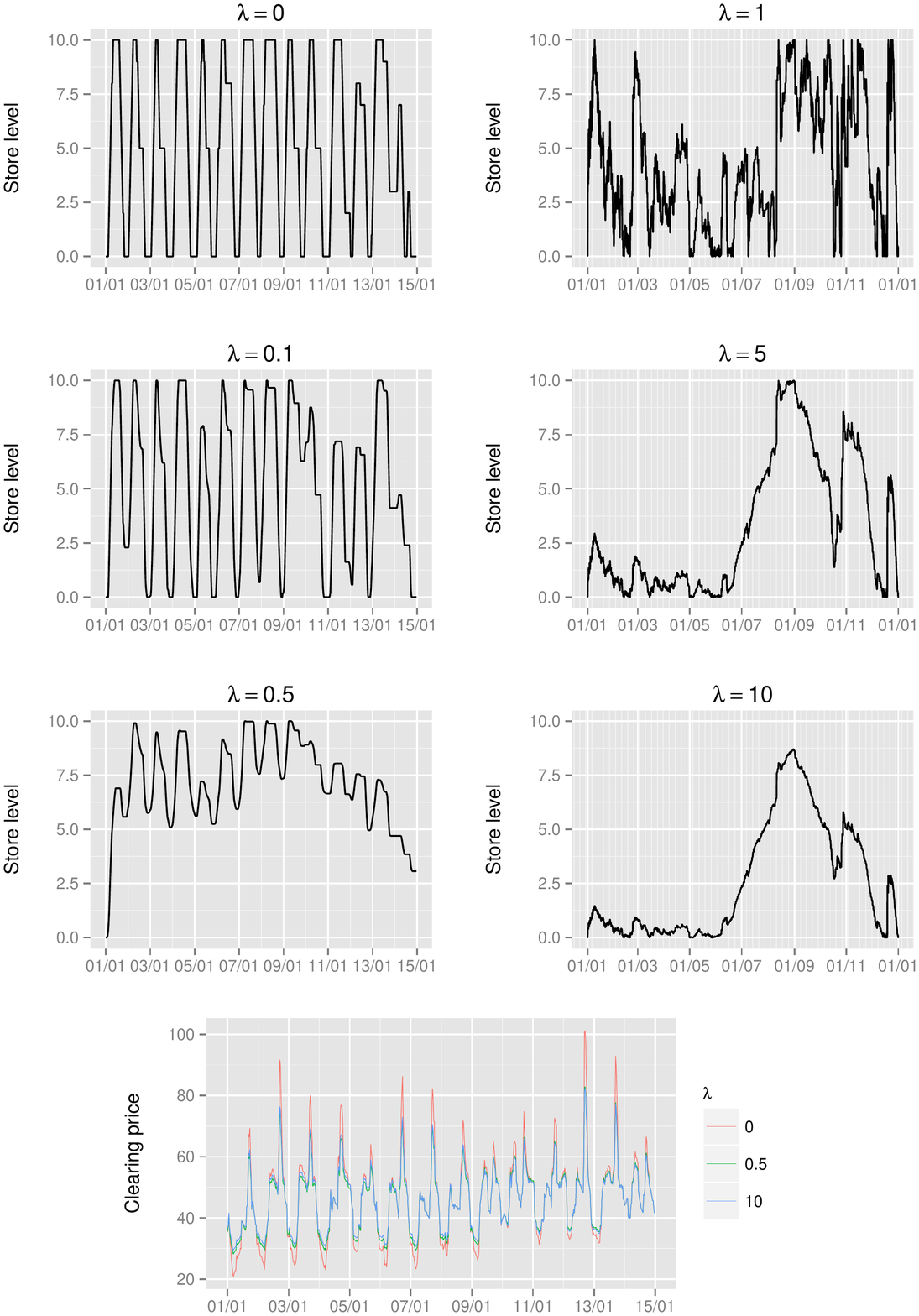}
  \caption{Single store: behaviour of store level and market clearing
    price (see text for a discussion of units) as the market impact
    factor~$\lambda$ is varied---equivalently the size of the store is
    varied.}
  \label{fig:lam}
\end{figure}

For values of $\lambda$ greater than $\lambda_{\max}\approx23$ the
volumes traded are such that neither the rate nor the capacity
constraints of the store are binding, so that for
$\lambda>\lambda_{\max}$ volumes traded are simply proportional to
$1/\lambda$.

The left panels of Figure~\ref{fig:capeff} show the effect on store
level---over the entire year---of decreasing the efficiency of the
store from $\epsilon=0.75$ (for which the store level is shown in red)
to $\epsilon=0.65$ (for which the store level is shown in blue), for
each of the larger values of $\lambda$ considered above, i.e.\ for
$\lambda=1$, $\lambda=5$ and $\lambda=10$.  The capacity of the store
is here kept at our base level of $E=10$.  Decreasing the efficiency
of the store reduces its ability to exploit the daily cycle of price
variation in a manner not dissimilar from that of increasing the
market impact~$\lambda$, so that again the volumes of daily trading
are reduced, while the store may continue to exploit its full capacity
on a seasonal basis---again for a very modest further gain.  We remark
also that reducing the efficiency of the store reduces the extent to
which it is able to smooth prices.

The right panels of Figure~\ref{fig:capeff} similarly show the
effect---again over the entire year and for the same three values of
$\lambda$---of increasing the capacity of the store from $E=10$ (for
which the store level is shown in red) to $E=20$ (for which the store
level is shown in blue).  The round trip efficiency of the store is
kept at $\epsilon=0.75$.  In each case it is seen that the daily
variation in the level of the store remains much the same as $E$ is
increased (since for these levels of $\lambda$ there is too much
market impact to make profitable greater volumes of daily trading,
except on occasions in the case $\lambda=1$).  However, for
$\lambda=1$ and for $\lambda=5$, as $E$ is increased the store is able
to make some (very modest) additional profit by varying slowly
throughout the year the general level at which it operates.  For
$\lambda=10$ the market impact is so great that the capacity
constraint $E=10$---and so also the capacity constraint $E=20$---is
never binding, so that in this case the increase in the capacity has
no effect.

\begin{figure}[!ht]
  \centering
  \includegraphics[scale=0.75]{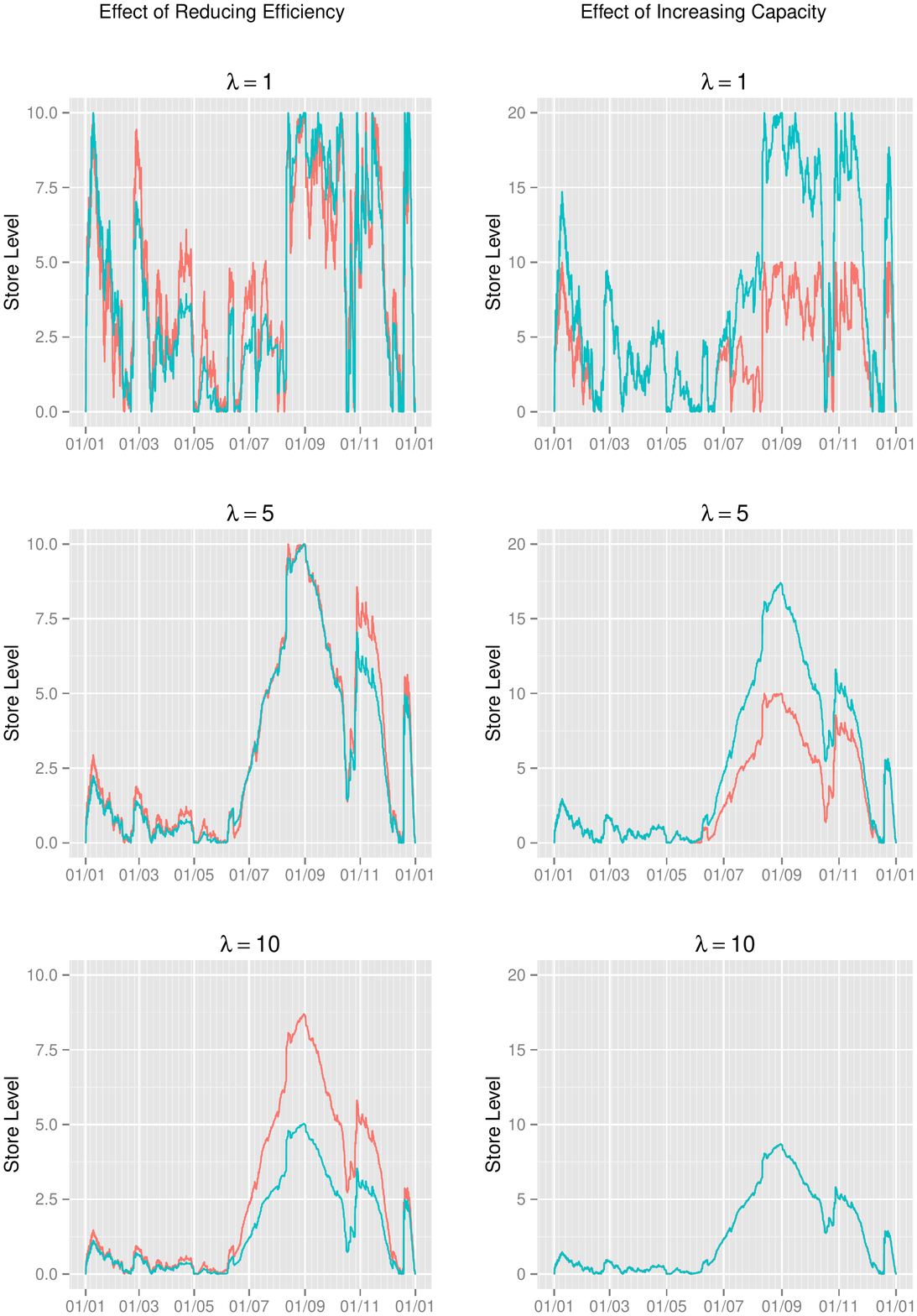}
  \caption{Single store: behaviour of the store level as the
    round-trip efficiency~$\epsilon$ is varied from $0.75$ to $0.65$
    (left panels) and as the capacity~$E$ is varied from $10$ to $20$
    (right panels), in each case for $\lambda=1$, $\lambda=5$ and
    $\lambda=10$.}
  \label{fig:capeff}
\end{figure}

\section{Competing stores in a market}
\label{sec:n-competing-stores}

In this section we discuss $n$ competing stores in a market, where it
is assumed that the objective of each store is to maximise its own
profit.  The optimal strategy of each store in general depends on the
activities of the remainder, and what happens depends on the extent to
which there is cooperation between the stores.  In the absence of any
such cooperation we might reasonably expect some form of convergence
over time to a Nash equilibrium, in which each store's strategy is
optimal given those of the others.  We first discuss briefly the
cooperative solution, primarily for the purpose of reference, before
considering the effect of market competition.

\subsection{The cooperative solution}
\label{sec:coop-solut}

Here the stores behave cooperatively so as to minimise their combined
cost
\begin{equation}
  \label{eq:12}
  \sum_{i=1}^n\sum_{t=1}^T h_i(x_t(S_i)) p_t\bigg(\sum_{k=1}^nh_k(x_t(S_k))\bigg),
\end{equation}
subject to the capacity constraints \eqref{eq:5} and rate
constraints~\eqref{eq:6}.  This is a generalisation to higher
dimensions of the single-store problem, and we do not discuss a
detailed solution here.  Note, however, that an iterative approach to
the determination of a solution may be possible.  Under our
assumptions on the price functions, the function of $S_1,\dots,S_n$
given by~\eqref{eq:12} is convex.  For any store~$j$, given the
levels~$S_i$ of the remaining stores $i\neq j$, the minimisation
of~\eqref{eq:12} in $S_j$ (subject to the above constraints) is an
instance of the single-store problem discussed in
Section~\ref{sec:single-store-market}---with cost functions modified
so as reflect the overall cost to all the stores of the actions of the
store~$j$.  This leads to the obvious iterative algorithm in
which~\eqref{eq:12} is minimised in $S_j$ for successive stores~$j$
until convergence is achieved.  However, the limiting value of
$(S_1,\dots,S_n)$, while frequently a global minimum,  is not
guaranteed to be so.

In the case where the stores have identical efficiencies one might
also consider the simplified single-store problem in which the
individual capacity constraints are summed and individual rate
constraints are summed.  If the solution to this, suitably divided
between the stores (i.e.\ with a fraction $\kappa_i$ of the optimal
flow assigned to each store $i$, where $\sum_{i=1}^n\kappa_i=1$), is
feasible for the original problem then it solves that problem.  One
case where this is true is where additionally the ratios $E_j/P_{Ij}$
and $E_j/P_{Oj}$ are the same for all stores~$j$; the solution to the
simplified single-store problem is then just divided among the stores
in proportion to their capacities to give the cooperative solution to
the $n$-store problem.

The impact of the stores on market prices and consumer surplus is
determined in a manner entirely analogous to that of
Section~\ref{sec:impact-store-prices}. 

\subsection{The competitive solution}
\label{sec:competitive-solution}

When stores compete there needs to be a mechanism whereby a clearing
price in the market is determined.  Here there are in principle
various possibilities according to the rules under which the market is
to operate.  We discuss some of these in
Section~\ref{sec:poss-models-comp}, making a formal link with the
various classical modes of competition in simple ``single shot in
time'' markets for balancing supply and demand in situations where
storage does not operate.  In the succeeding sections we look in
particular at what happens when stores bid \emph{quantities}, i.e.\ at
Cournot models of competition.

\subsubsection{Possible models of competition}
\label{sec:poss-models-comp}

Consider first the case $T=2$, and assume for simplicity that the
stores are perfectly efficient.  Suppose that each store $k$ buys and
then sells $q_k$ (positive or negative), and that this results in a
price differential of $p$ (the clearing price at time 2 less that at
time 1) so that each store $k$ makes a profit $pq_k$.  We might
consider the situation where, in a precise analogue of the
\emph{supply function bidding} of Klemperer and Meyer~\cite{KM}, each
store $k$ declares, for each possible value of $p$, a value $S_k(p)$
which it contracts to buy at time 1 and then sell at time 2 if the
clearing prices at those times are set such that the price
differential is $p$.  If each ``supply function'' $S_k$ is a
nondecreasing function of $p$, the auctioneer then chooses the
clearing prices $p_1$ and $p_2$ such that
\begin{align}
  R_1(p_1) & = \sum_k S_k(p) \label{eq:13}\\
  R_2(p_2) & = -\sum_k S_k(p) \label{eq:14}\\
  p_2 - p_1 & = p \label{eq:15},
\end{align}
where, for $t=1,2$, $R_t$ is the residual supply function defined in
Section~\ref{sec:model}.

Assume that the residual supply functions~$R_t$ are strictly
increasing.  The system of equations~\eqref{eq:13}--\eqref{eq:15} is
easily seen to have a unique solution (provided the supply
functions~$S_k$ are such that one exists at all): suppose that, as $p$
varies, $p_1$ and $p_2$ are chosen as functions of $p$ such that
$p_2-p_1=p$ and $R_2(p_2)=-R_1(p_1)$; then, as $p$ increases,
$\sum_k S_k(p)$ increases while $R_1(p_1)$ decreases, and at the
unique value of $p$ such that we have equality between these two
quantities the above system of equations~\eqref{eq:13}--\eqref{eq:15}
is satisfied.

Mathematically, this situation is no different from that of the
classical ``one-shot'' supply function bidding of Klemperer and
Meyer~\cite{KM}.  This was further studied in applications to energy
markets by Green and Newbery~\cite{GN} and by Bolle~\cite{Bol}, and
subsequently by many others---see in particular Anderson and
Philpott~\cite{AP2002}, and the very comprehensive review by Holmberg
and Newbery~\cite{HN}.  In such supply function bidding suppliers (for
example, electricity generators) submit nondecreasing supply functions
to a market in which there is also a nonincreasing demand function,
the market clearing price being that at which the total supply equals
the total demand.  The behaviour of such supply function bidding is
considered in~\cite{KM}, in particular the existence and uniqueness of
Nash equilibria.  In practice one might well wish to restrict the
allowable sets of supply functions which suppliers are permitted to
bid (see Johari and Tsitsiklis~\cite{JT}) so as to achieve
economically acceptable solutions.  Two extreme cases are the
classical situations where either suppliers may bid \emph{prices} at
which they are prepared to supply any amount of the commodity to be
traded---corresponding to ``vertical'' supply functions and leading to
a \emph{Bertrand equilibrium}, or else suppliers may bid
\emph{quantities} which they are prepared to supply at whatever price
clears the market---corresponding to horizontal supply functions and
leading to a \emph{Cournot equilibrium}.  In the former case, at the
Nash equilibrium, the one supplier who is able to offer the lowest
price corners the market (and, in the case of symmetric suppliers,
makes zero profit).  In the latter case, modest profits are to be
made, but the total profit of all the suppliers decreases rapidly as
their number increases---as is seen also in our results for storage
models below.

It is difficult to find a sensible and realistic way of extending the
concept of general supply function bidding to competition amongst
stores operating over more than two time periods---the dimensionality
of the space in which the supply functions would then live is too
high, and the set of possibilities for market clearing mechanisms is
too complex.  Nor is it realistic to consider the situation where
stores bids prices, since as indicated above, profits are then
typically too small for stores to be able to recover their set-up
costs.  We therefore restrict our attention to the case where stores
bid quantities---as seems to be the case where elsewhere in the
literature market competition between stores is considered (see, for
example, Sioshansi~\cite{Sio2014}).  Here the Nash equilibria are
Cournot equilibria and the profits made by the stores at such
equilibria may be expected to provide reasonable upper bounds on such
profits as might be made in practice---for a review in the context of
``one-shot in time'' markets again see Holmberg and Newbery~\cite{HN}.

\subsubsection{General convex cost functions}
\label{sec:general-convex-cost}

We consider stores bidding quantities as above and look for Nash
(Cournot) equilibria.  A (pure strategy) Nash equilibrium is then a
set of vectors $(S_1,\dots,S_n)$ such that the strategy~$S_j$ of each
store~$j$ (i.e.\ the vector of quantities traded over time by that
store) is optimal given the strategies $S_i$, $i\neq j$, of the
remaining stores; thus the vector~$S_j$ solves the optimisation
problem $\mathbf{P}_j$ (defined by the remaining vectors $S_i$,
$i\neq j$) of Section~\ref{sec:model}.  Equivalently, at a Nash
equilibrium, the vector~$S_j$ minimises the function~\eqref{eq:12}
subject to the constraints~\eqref{eq:5} and \eqref{eq:6} and with the
values of the vectors $S_i$, $i\neq j$, held constant.

Broadly what happens at such an equilibrium is that stores will buy
and sell more than at the cooperative solution, since each store gains
for itself the benefits of so doing, while the corresponding costs are
shared out among all stores.  In particular consider $n$ identical
competing stores with nonbinding capacity and rate constraints, but
with common given starting and finishing levels; for the moment assume
further that they have round-trip efficiencies $\epsilon=1$, and that
the price functions~$p_t$ are differentiable.  For each store~$k$ and
for each time~$t$, write $x_{kt}=x_t(S_k)$.  At the symmetric Nash
equilibrium, and for each store~$j$, there are equalised over time~$t$
the partial derivatives with respect to $x_{jt}$ of the functions
$x_{jt}p_t\bigl(\sum_{k=1}^nx_{kt}\bigr)$.  (For $n=1$ these are just
the derivatives of the cost functions seen by the store.)  It is
straightforward to show that the convexity of these functions ensures
that in general unit prices received by the store at those times when
it is selling are higher than unit prices paid by the store at those
times when it is buying, and so the store is able to make a strictly
positive profit.  However, as $n$ becomes large the above partial
derivatives tend to the price functions
$p_t\bigl(\sum_{k=1}^nx_{kt}\bigr)$ so that, in the limit as
$n\to\infty$, prices become equalised over time and the stores no
longer make any profit.  As earlier, the intuitive explanation is that
in the limit the stores become price takers and any individual store
is able to exploit any inequality over time in market clearing prices
so as to increase its profit.  Thus at the Nash equilibrium market
clearing prices are equalised over time and stores are unable to make
any profit.  It is easy to see that essentially the same result holds
when round-trip efficiencies are less than one.  In the case of
linearised price functions we quantify this result further in
Theorem~\ref{thm:ucs}.

More generally the impact on prices of competition between stores, in
comparison to the cooperative solution, is to further reduce the price
variation between the different times over which the stores operate.
Arguing as in Section~\ref{sec:impact-store-prices}, one would
typically expect such increased competition to lead to a further
increase in consumer surplus.  However, again this need not always be
the case.

\paragraph{Existence and uniqueness of Nash equilibria.}

The following result shows the existence of a (pure strategy) Nash
equilibrium.

\begin{theorem}\label{thm:ne_exists}
  Under the given assumptions on the price functions~$p_t$, there
  exists at least one Nash equilibrium.
\end{theorem}

\begin{proof}
  The assumptions on the price functions~$p_t$ guarantee convexity of
  the cost functions defined by~\eqref{eq:4}.  We assume first that
  the price functions are such that these cost functions are strictly
  convex.  Write $S = (S_1,..., S_n)$ where each $S_j$ is the strategy
  over time of store $j$.  Let $\mathcal{S}$ be the set of all
  possible $S$; note that $\mathcal{S}$ is convex and compact.  Define
  a function $f:\mathcal{S}\rightarrow\mathcal{S}$ by
  $f(S)=(f_1(S),\dots,f_n(S))$ where each $f_j(S)$ minimises the
  function $G_j(\,\cdot\,;\,S_1,\dots,S_{j-1},S_{j+1},\dots,S_n)$
  given by \eqref{eq:4} subject to the constraints~\eqref{eq:5} and
  \eqref{eq:6}, i.e.\ $f_j(S)$ is the best response of store $j$ to
  $(S_1,\dots,S_{j-1},S_{j+1},\dots,S_n)$.  It follows from the strict
  convexity assumption that each $f_j(S)$ is uniquely defined.

  Now suppose that a sequence $(S^{(n)})$ in $\mathcal{S}$ is such
  that $S^{(n)}\to S$ as $n\to\infty$.  Then, for each~$j$, the
  functions
  $G_j(\,\cdot\,;\,S^{(n)}_1,\dots,S^{(n)}_{j-1},S^{(n)}_{j+1},\dots,S^{(n)}_n)$
  (of $S_j$) converge uniformly to the continuous and strictly convex
  function $G_j(\,\cdot\,;\,S_1,\dots,S_{j-1},S_{j+1},\dots,S_n)$, so
  that also $f_j(S^{(n)})\rightarrow f_j(S)$.  Hence the function $f$
  is itself continuous.  Thus by the Brouwer fixed point theorem there
  exists $S = f(S)$, which by definition is a (Cournot) Nash
  equilibrium.

  In the case where the price functions~$p_t$ are such that the cost
  functions given by \eqref{eq:4} are convex but not strictly so, we
  may consider a sequence of modifications to the former, tending to
  zero and such that we do have strict convexity of the corresponding
  cost functions.  Compactness ensures that the corresponding Nash
  equilibria converge, at least in a subsequence, to a limit which
  straightforward continuity arguments show to be a Nash equilibrium
  for the problem defined by the unmodified price functions.
\end{proof}

In general the uniqueness of any Nash equilibrium is unclear.
However, we show in Section~\ref{sec:quadr-cost-funct} that, under a
linear approximation to the price functions, the Nash equilibrium
\emph{is} unique.

The proof of Theorem~\ref{thm:ne_exists} also suggests an iterative
algorithm to identify possible Nash equilibria---analogous to the
algorithm suggested in Section~\ref{sec:coop-solut}.  Given any $S$
the determination of each $f_j(S)$ introduced in the above proof
requires only the solution of single-store optimisation problem, which
may be achieved as described in, for example,~\cite{CFGZ}).  Hence,
starting with any $S^{(0)}$, we may construct a sequence
$\{S^{(n)}\}_{n\ge0}$ such that $S^{(n)}=f(S^{(n-1)})$.  Then, as in
the above proof, any limit~$S$ of the sequence $\{S^{(n)}\}$ satisfies
$S=f(S)$ and hence constitutes a Nash equilibrium.  Different starting
points $S^{(0)}$ may be tried, but, in the case of nonuniqueness, there
is of course no guarantee that all Nash equilibria will be found.

Even under our given assumptions on the price functions~$p_t$
the general characterisation of Nash equilibria seems difficult.
The following theorem gives a monotonicity result.

\begin{theorem}
  \label{thm:op}
  Consider $n$ competing stores with identical rate constraints and
  efficiencies and whose starting levels and finishing levels are
  ordered by their capacity constraints.
  Then, at any Nash equilibrium~$S^*=(S^*_1,\dots,S^*_n)$, the levels
  of the stores are at all times ordered by their capacity
  constraints.
\end{theorem}
\begin{proof}
  Let $(\mu^*_1,\dots,\mu^*_n)$ be the set of vectors (Lagrange
  multipliers) associated with the Nash equilibrium
  $S^*=(S^*_1,\dots,S^*_n)$ as defined by Proposition~\ref{prop:1}.
  It follows from (ii) of that proposition that, for any $t$, and any
  $i$, $j$,
  \begin{equation}
    \label{eq:16}
    \mu^*_{it} \ge \mu^*_{jt}
    \quad \Longleftrightarrow \quad
    x_t(S^*_i) \ge x_t(S^*_j). 
  \end{equation}
  Suppose now that the assertion of the theorem is false.  Then there
  exist $i$, $j$ with $E_i<E_j$ and some $t_0$ such that
  \begin{equation}
    \label{eq:17}
    x_{t_0}(S^*_i) > x_{t_0}(S^*_j), \qquad S^*_{it_0} > S^*_{jt_0}.
  \end{equation}
  It now follows by induction that, for all $t'\ge t_0$,
  \begin{equation}
    \label{eq:18}
    x_{t'}(S^*_i) \ge x_{t'}(S^*_j), \qquad S^*_{it'} > S^*_{jt'},
    \qquad \mu^*_{it'} \ge \mu^*_{jt'}.
  \end{equation}
  That \eqref{eq:18} is true for $t'=t_0$ follows from \eqref{eq:16} and
  \eqref{eq:17}.  Suppose now that \eqref{eq:18} is true for some
  particular $t'\ge t_0$.  It then follows from
  Proposition~\ref{prop:1} that the condition $S^*_{it'}>S^*_{jt'}$
  implies $\mu^*_{i,t'+1}\ge\mu^*_{j,t'+1}$; hence, by \eqref{eq:16},
  $x_{t'+1}(S^*_i)\ge x_{t'+1}(S^*_j)$ and so finally
  $S^*_{i,t'+1}>S^*_{j,t'+1}$.  However, this contradicts the
  assumption $S^*_{iT}\le S^*_{jT}$.
\end{proof}

\subsubsection{Quadratic cost functions (i.e.\ linearised price functions)}
\label{sec:quadr-cost-funct}

We can make considerably more progress in the case of the linear
approximation to the price functions given by
equation~\eqref{eq:1}, where we again assume that, for each $t$, we
have $\bar p_t=p_t(0)>0$, $p_t'=p_t'(0)\ge0$, and that the
function~$p_t$ remains positive over the range of possible values of
its argument (so that our standing assumptions on the functions~$p_t$
are satisfied).  This linearisation~\eqref{eq:1} is a reasonable
approximation when storage facilities are sufficiently large as to
have an impact on market prices, but are not so very large as to
require a more sophisticated price function.  The main reason for
greater analytical tractability in this case is that for a set of
vectors $(S_1,\dots,S_n)$ to a be Nash equilibrium is then equivalent
to the requirement that they minimise a given convex function.  In
particular we have the following result.


\begin{theorem}
  \label{thm:une}
  Given the price functions~\eqref{eq:1}, there always exists a unique
  Nash equilibrium.
\end{theorem}
\begin{proof}
  It follows from \eqref{eq:1} and \eqref{eq:4} that the requirement
  that a set of vectors $(S_1,\dots,S_n)$ be a Nash equilibrium is
  equivalent to the requirement that, for each store~$j$, given the
  policies $S_i$, $i\neq j$, being operated by the remaining stores,
  the vector $S_j$ minimises the total cost
  \begin{equation}
    \label{eq:19}
    \sum_{t=1}^T h(x_t(S_j)) \bigg(\bar p_t+p'_t\sum_{i=1}^n h(x_t(S_i))\bigg),
  \end{equation}
  subject to the capacity and rate constraints on store $j$ given by
  \eqref{eq:5} and \eqref{eq:6}.  Now note that this is further
  equivalent to the requirement that the set of vectors
  $(S_1,\dots,S_n)$ minimises the strictly convex function
  \begin{equation}
    \label{eq:20}
       \sum_{t=1}^T
    \left[
      \bar p_t \sum_{i=1}^nh_i(x_t(S_i)) + \frac{1}{2}p'_t
      \left(
        \sum_{i=1}^nh_i(x_t(S_i))^2 + \bigg(\sum_{i=1}^nh_i(x_t(S_i))\bigg)^2
      \right)
    \right]
  \end{equation}
  subject to the constraints \eqref{eq:5} and \eqref{eq:6} being
  satisfied for all $j$.  Further since this minimum is also to be
  taken over a compact set, its existence and uniqueness---and hence
  that of the Nash equilibrium---follows.
\end{proof}

Theorem~\ref{thm:scale} below, which is a scaling result, reduces the
optimisation problem (the determination of the Nash equilibrium) for
$n$ identical competing stores to that of the corresponding problem
for an appropriately redimensioned single store.

\begin{theorem}
  \label{thm:scale}
  Given the price functions~\eqref{eq:1} and a common
  efficiency~$\epsilon$, for each $n\ge1$, consider $n$ identical
  competing stores with common capacity $E^{(n)}$, common rate input
  and output constraints $P^{(n)}_I$ and $P^{(n)}_O$, and common
  starting and finishing levels $S^{(n)}_0$ and $S^{(n)}_T$
  respectively, where we have
  \begin{gather*}
    E^{(n)} = 2E^{(1)}/(n+1),\\
    P^{(n)}_I = 2P_I^{(1)}/(n+1), \qquad P^{(n)}_O = 2P_O^{(1)}/(n+1),\\
    S^{(n)}_0 = 2S_0^{(1)}/(n+1), \qquad S^{(n)}_T = 2S_T^{(1)}/(n+1). 
  \end{gather*}
 For each~$n$, let $S^{(n)}=(S^{(n)}_1,\dots,S^{(n)}_T)$ be the
  common policy over time of each of the stores at the unique and
  necessarily symmetric competitive Nash equilibrium.  Then, at this
  equilibrium and at each time $t$, the quantity traded by each store
  in the $n$-store problem is $2/(n+1)$ times the quantity traded in
  the single store problem, i.e.\
  $h(x_t(S^{(n)}))=2h(x_t(S^{(1)}))/(n+1)$.
\end{theorem}
\begin{proof}
  It follows from Theorem~\ref{thm:une} that, for each $n$, $S^{(n)}$
  minimises the strictly convex function
  \begin{equation}
    \label{eq:21}
    n \sum_{t=1}^T
    \left(
      \bar p_t h(x_t(S^{(n)})) + \frac{1}{2}(n+1)p'_t h(x_t(S^{(n)}))^2
    \right)
  \end{equation}
    subject to the capacity constraints
  \begin{displaymath}
    S^{(n)}_0=S^*_0/(n+1), \quad S^{(n)}_T=S^*_T/(n+1),
    \quad 0 \le S^{(n)}_t\le E/(n+1), \ \, 1 \le t \le T-1,
  \end{displaymath}
  and the rate constraints
  \begin{displaymath}
     -P_I/(n+1) \le x_t(S^{(n)}) \le  P_O/(n+1),
    \quad 1 \le t \le T.
  \end{displaymath}
  The substitution $z_t=2(n+1)x_t(S^{(n)})$, for $t=1,\dots,T$, yields
  a single store
  minimisation problem which is independent of $n$ (apart from a
  factor $2n/(n+1)$ in the objective~\eqref{eq:21}) so that, for each
  $t$, $x_t(S^{(n)})$ (and so also $h(x_t(S^{(n)}))$) is proportional
  to $1/(n+1)$, so that the required result is now immediate.
\end{proof}

\begin{remark}
  The reduction in Theorem~\ref{thm:scale} (for linear price
  functions) of the problem for $n$ identical stores to a single store
  problem, allows also the application of the various sensitivity
  results of Sections~\ref{sec:sensitivity-results} and
  \ref{sec:impact-cons-surpl}.
\end{remark}

Theorem~\ref{thm:ucs} below shows that $n$ \emph{unconstrained} stores
(with identical efficiencies) in competition make very much less
profit in total than a single unconstrained store operating in the
same market.

\begin{theorem}
  \label{thm:ucs}
  Given the price functions~\eqref{eq:1} and a common
  efficiency~$\epsilon$, consider $n$ stores subject to neither
  capacity nor rate constraints.  Suppose further that the stores have
  a common starting level $S^*_0$ and the same common finishing level
  $S^*_T=S^*_0$, and that this level is sufficiently large that, at
  the (unique and necessarily symmetric) Nash equilibrium, the stores
  never empty.  Then, at this equilibrium, the quantity traded per
  store is proportional to $1/(n+1)$ and the profit per store is
  proportional to $1/(n+1)^2$.
\end{theorem}
\begin{proof}
  The first assertion of the theorem may be deduced from the scaling
  result of Theorem~\ref{thm:scale}, and that theorem might be
  extended to enable also the second assertion of the present theorem
  to be deduced.  However, we use instead the argument below, which
  also explicitly identifies the behaviour of the stores.

  Write $\bar S=(\bar S_0,\dots,\bar S_T)$ (where
  $\bar S_T = \bar S_0 = S^*_0$) for the common policy over time of
  each of the stores at the Nash equilibrium.  It now follows from
  Theorem~\ref{thm:une} and the minimisation of the
  function~\eqref{eq:20} subject to the constraint
  \begin{equation}
    \label{eq:22}
    \bar S_T = \bar S_0,
  \end{equation}
  that this equilibrium is given by
 \begin{equation}
    \label{eq:23}
    x_t(\bar S) =
    \begin{cases}
      \dfrac{\lambda - \bar p_t}{(n+1)p'_t}, & \quad \bar p_t < \lambda\\
      0, & \quad \lambda \le \bar p_t \le \dfrac{\lambda}{\epsilon}\\
      \dfrac{\lambda - \epsilon \bar p_t}{(n+1)\epsilon^2 p'_t}, &
      \quad \bar p_t \ge \dfrac{\lambda}{\epsilon}.
    \end{cases}
  \end{equation}
  for some Lagrange multiplier~$\lambda$ such that \eqref{eq:22} is
  satisfied.  Note, in particular, that $\lambda$ is independent of
  $n$.  Thus, as $n$ varies, we have again that $(x_1(\bar
  S),\dots,x_T(\bar S))$ is proportional to $1/(n+1)$ as required.  It
  follows also from \eqref{eq:23} (by checking separately each of the
  three cases there) that, for all~$t$,
  \begin{equation}
    \label{eq:24}
    h(x_t(\bar S)) (\bar p_t + (n+1)p'_t h(x_t(\bar S))) = \lambda x_t(\bar S).
  \end{equation}
  It follows from~\eqref{eq:19} and from \eqref{eq:24} that, at the
  Nash equilibrium, each store~$j$ incurs a total cost (the negative
  of its profit) equal to
  \begin{align*}
    \sum_{t=1}^T h(x_t(\bar S)) (\bar p_t + np'_t h(x_t(\bar S)))
    & = \sum_{t=1}^T \lambda x_t(\bar S) - p'_t h(x_t(\bar S))^2\\
    & = - \sum_{t=1}^T p'_t h(x_t(\bar S))^2,
  \end{align*}
  where the first equality above follows from~\eqref{eq:24} and the
  second from \eqref{eq:22}.  Since, as $n$ varies, $(h(x_1(\bar
  S)),\dots h(x_T(\bar S)))$ is proportional to $1/(n+1)$, the
  required result for the profit of each store follows.
\end{proof}

Note that, under the conditions of the above theorem, the total
quantity traded by the $n$ stores (at each instant in time) is
$2n/(n+1)$ times that traded by a single store, while the total profit
made by the $n$ stores is $4n/(n+1)^2$ times that made by a single
store.  Thus we here quantify our earlier assertion of the
Introduction that competing stores overtrade (for the reasons already
discussed there) in comparison to the cooperative solution; as
$n\to\infty$ their combined profit decreases towards zero.  Clearly
also, were the stores subject to capacity or rate constraints, their
ability to negatively impact on each other would be less---as in the
example below.

\subsection{Example}
\label{sec:examples}

We consider again the half-hourly Market Index Price data for Great
Britain throughout 2014, as introduced in the example of
Section~\ref{sec:example}.  We again let the price function be as
given by~\eqref{eq:11} and (without loss of generality as explained in
Section~\ref{sec:example}) take the market impact factor $\lambda=1$.
We consider $n=1,2,3$ identical stores in competition, each with a
round-trip efficiency $\epsilon=0.75$.  For the single-store case
$n=1$, we take $E=10$ and common input and output rate constraint
$P=1$; for $n=2$ we take $E=5$ and $P=1/2$ for each of the two stores,
and for $n=3$ we take $E=10/3$ and $P=1/3$ for each of the three
stores.  Thus the total storage available in each case is the same.
The values of $E$ and $P$ are chosen so that the constraints on the
stores are not so severe as to force essentially identical combined
behaviour of the stores for each of the three values of $n$
considered; nor are they so lax that the stores behave as if they were
unconstrained as considered in Theorem~\ref{thm:ucs}.  For each $n$,
we consider the unique Nash equilibrium in which each of the $n$
stores optimises its behaviour (minimises its cost) over the entire
year subject to the constraints of starting and finishing empty, and
(for $n>1$) given the behaviour of the remaining store(s).

In the units of the example---for a discussion of which again see
Section~\ref{sec:example}---the total profits made throughout the year
by the $n$ stores are 4096 for $n=1$, 3733 for $n=2$ and 3267 for
$n=3$.  For each of the latter two cases, if the stores were to
cooperate instead of competing, they would make the same total profit
as in the single store case.  Thus the decrease in total profit is
again due to the effects of competition.  However, note that as $n$
increases through the above three values the total profit decreases at
a rate which is slower than that in the case of unconstrained stores,
as given by Theorem~\ref{thm:ucs}.

Figure~\ref{fig:competition} shows the total level of the $n=1,2,3$
stores and the corresponding market clearing prices (again in the
units of the example) over the first two weeks of the year.  The upper
panel of the figure clearly shows that $n=2$ and $n=3$ competing
stores consistently overtrade in relation to the case $n=1$
(corresponding to the cooperative solution).  The lower panel shows
the extent to which competition between multiple stores smooths market
clearing prices, which is of course associated with the reduction in
overall profits.  The times of maximum store activity correspond to
the peaks and troughs of the market clearing price and it is these
peaks and troughs which are smoothed by the competition.  Note also
that, because the round-trip efficiency $\epsilon=0.75$ is
significantly less than $1$, there are significant periods of during
which the stores neither buy nor sell.

\begin{figure}[!ht]
  \centering
  \includegraphics[scale=0.76]{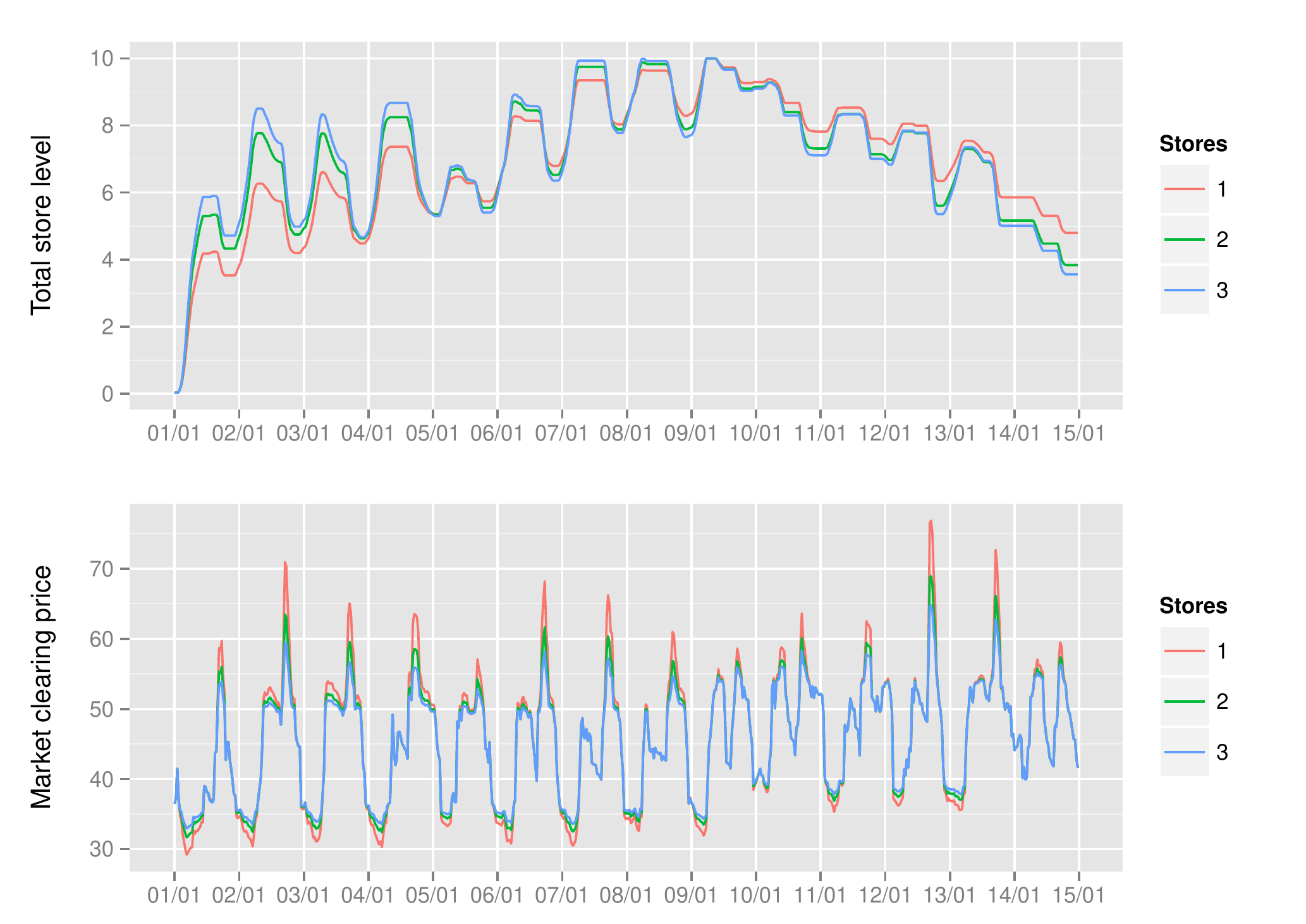}
  \caption{Total store level and market clearing price for each of
    $n=1,2,3$ stores in competition.}
  \label{fig:competition}
\end{figure}


  


\section{Variant problems}
\label{sec:variant-problems}

Heretofore we have considered the optimal control of stores where the
objective of each has in general been to maximise its own profit,
obtained through price arbitrage over time.  Such behaviour has a
variable effect on both producers (in the case of energy the
generators) and consumers.  However, a store may alternatively be used
to maximise the benefit either to the consumers (i.e.\ to society, if
the generators are excluded from the latter), or to the generators, or
to society as a whole.  We consider briefly each of these
possibilities, so as to show that in each case essentially the same
mathematical model applies---and hence also both the form of its
solution and insights into the effects of competitive behaviour.

\paragraph{One or more stores owned by the consumers.}
\label{sec:one-or-more}

Suppose that a single store is notionally owned by the consumers
(i.e.\ by society if the latter excludes the generators).  Here the
problem is to use it so as to maximise the benefit to society.  If at
each time $t$ an amount $x_t$ (positive or negative) is placed in the
store, then this has a total consumer cost (again positive or
negative) which is the sum of the extra payment to the generator plus
the reduction in consumer surplus due to the market impact of the
activity of the store (the reduction in consumer surplus being zero in
the case where the generator has a flat supply function).  The vector
$x=(x_1,\dots,x_T)$ should then be chosen so as to minimise this total
cost, and that is just an instance of the mathematical problem
considered in Section~\ref{sec:single-store-market} and for which
Proposition~\ref{prop:1} describes the form of the optimal solution.
Note that in the case where the generator's prices are constant over
both volume and time, the store, even if perfectly efficient, is of
zero value.

\paragraph{One or more stores  owned by the generator.}
\label{sec:one-or-more-1}

Now suppose that a store is owned by a generator, and is used by the
latter with the intention of maximising its own total profit.  Thus
if, at each time~$t$, an amount $x_t$ (positive or negative) is placed
in the store, then this has a cost to the generator which is simply
that of producing it; further, if (at that time) the generator's
production costs are nonlinear, the generator will re-optimise the
amount supplied to the market, thereby affecting its profit from that
activity; hence we may determine the total cost to the generator of
the action $x_t$.  The vector $x=(x_1,\dots,x_T)$ may then be chosen
so as to minimise this total cost (i.e.\ to maximise profit), and this
is again just an instance of the problem considered in
Section~\ref{sec:single-store-market}.  Again in the case where the
generator's production costs are linear and constant over time, the
store, even if perfectly efficient, is of zero value.

\paragraph{Both generators and stores  owned by society.}
\label{sec:both-gener-stor}

Finally suppose that both the generator(s) and any store are owned by
the consumers, i.e.\ by society, and managed jointly so as to maximise
the benefit to society.  In the absence of the store, the generator's
supply function may be replaced by its (inverse) cost function i.e.\
that function which gives the amount which may be (just) economically
supplied as a (generally increasing) function of unit price; the point
of intersection of this function with the demand function gives the
optimal price, and the (optimised) benefit to society is the consumer
surplus at that price.  The introduction of the store now modifies
this theory in a manner entirely analogous to that in the earlier case
where just the store is owned by society.





\section{Conclusions}
\label{sec:conclusions}

In the present paper we have considered how storage, operating as a
price maker within a market environment, may be optimally operated
over an extended or indefinite period of time.  The optimality
criterion may be that of maximising the profit over time of the
storage itself, where this profit results from the ability of the
storage to exploit differences in market clearing prices at different
times.  Alternatively it may be that of minimising over time the cost
of generation, or of maximising consumer surplus or social welfare.
In all cases there is calculated for each successive step in time the
cost function measuring the total impact of whatever action (amount to
buy or sell) is taken by the storage.  The succession of such cost
functions provides the appropriate information to the storage as to
how to behave over time, forming the basis of the appropriate
mathematical optimisation problem.  Further optimal decision making,
even over a very long time period, usually depends on a knowledge of
costs over a relatively short running time horizon---in the case of
the storage of electrical energy typically of the order of a day or
so.  We have also studied the various economic impacts---on market
clearing prices, consumer surplus and social welfare---of the
activities of the storage.  Where these impacts are considered
undesirable, the remedy is again the modification of the successive
cost signals supplied to the storage.  We have given examples based on
real Great Britain market data.

We have be particularly concerned to study competition between
multiple stores, where the objective of each store is to maximise its
own income given the activities of the remainder.  We have shown that
at the Nash equilibrium---with respect to Cournot
competition---multiple stores of sufficient size collectively erode
their own abilities to make profits: essentially each store attempts
to increase its own profit over time by overcompeting at the expense
of the remainder.  We have quantified this in the case of linear price
functions, and again given examples based on market data.

\section*{Acknowledgements}
\label{sec:acknowledgements-1}

The authors wish to thank their co-workers Frank Kelly and Richard
Gibbens for very helpful discussions during the preliminary part of
this work.  They are also grateful to the Isaac Newton Institute for
Mathematical Sciences in Cambridge for their funding and hosting of a
number of most useful workshops to discuss this and other mathematical
problems arising in particular in the consideration of the management
of complex energy systems.  Thanks also go to members of the IMAGES
research group, in particular Michael Waterson, Robert MacKay, Monica
Giulietti and Jihong Wang, for their support and useful discussions.
The authors are further grateful to National Grid plc for additional
discussion on the Great Britain electricity market, and finally to the
Engineering and Physical Sciences Research Council for the support of
the research programme under which the present research is carried
out.

\bibliography{storage_refs}
\bibliographystyle{plain}

\end{document}